\documentclass[a4paper]{article}
\usepackage[utf8]{inputenc}
\usepackage{lmodern}
\usepackage[OT1]{fontenc} 
\usepackage[english]{babel}
\usepackage{amsmath,amsthm,amssymb}
\usepackage{mathrsfs} 
\usepackage{graphicx}
\usepackage{framed}
\usepackage{xcolor}

\theoremstyle{plain}
\newtheorem{theorem}{Theorem}[section]
\newtheorem*{theorem*}{Theorem}
\newtheorem{lemma}[theorem]{Lemma}

\newtheorem{proposition}[theorem]{Proposition}
\newtheorem*{proposition*}{Proposition}

\newtheorem*{conjecture*}{Conjecture}

\theoremstyle{definition}
\newtheorem{definition}[theorem]{Definition}

\theoremstyle{remark}

\DeclareMathOperator{\supp}{supp}

\newcommand{\abs}[1]{\left\lvert #1 \right\rvert}
\newcommand{\norm}[1]{\left\lVert #1 \right\rVert}

\newcommand{\R}{\mathbb{R}}
\newcommand{\C}{\mathbb{C}}
\newcommand{\Z}{\mathbb{Z}}
\newcommand{\N}{\mathbb{N}}

\newcommand{\DOI}{B} 

\newcommand{\qnorm}[1]{{\norm{ #1 }_{C^7}}}
\newcommand{\qspace}{{C^7_c(\DOI)}}
\newcommand{\qspacen}[1]{{C^{#1}_c(\DOI)}}
\newcommand{\qsmoothness}{7}

\title{Well-posedness of the Goursat problem and stability for point source inverse backscattering}
\author{Eemeli Bl{\aa}sten\thanks{HKUST Jockey Club Institute for
    Advanced Study, Hong Kong University of Science and Technology,
    Hong Kong SAR. Email: \texttt{iaseemeli@ust.hk}}}

\begin{document}
\maketitle

\begin{abstract}
  We show logarithmic stability for the point source inverse
  backscattering problem under the assumption of angularly controlled
  potentials. Radial symmetry implies H\"older stability. Importantly,
  we also show that the point source equation is well-posed and also
  that the associated characteristic initial value problem, or Goursat
  problem, is well-posed. These latter results are difficult to find
  in the literature in the form required by the stability proof.
\end{abstract}

\medskip
{\flushleft
MSC classes: 35R30, 78A46, 35A08, 35L15

}

\section{Introduction}
For a potential function $q$ supported inside the unit disc $B$ in
$\R^3$ and a point $a$ consider the point source problem
\begin{alignat}{2}
  (\partial_t^2 - \Delta - q)U^a(x,t) & = \delta(x-a,t), &\qquad&
  x\in\R^3, t\in\R, \label{EQ1}\\ U^a(x,t) & = 0, &\qquad& x\in\R^3,
  t<0. \label{EQ2}
\end{alignat}
We define the point source backscattering data as the function
$(a,t)\mapsto U^a(a,t)$. This paper has two goals: to prove the
well-posedness of \eqref{EQ1}--\eqref{EQ2}, and then to solve the
inverse problem of determining $q$ from the point source
backscattering data $U^a(a,t)$ with $a\in\partial B$ and $t>0$.

The ordinary inverse problem of backscattering for arbitrary
potentials is a major open problem. In it the scattering amplitude
$A(\hat x, \theta, k)$ is measured for frequencies $k\in\R_+$,
incident plane-wave directions $\abs{\theta}=1$, and measurement
direction $\hat x=-\theta$. The question is whether such data
corresponds to a unique potential $q$. This question has been solved
in the time-domain for an admissible class of potentials in
\cite{RU1}. For a more in-depth review of earlier results please refer
to \cite{MU}.

Traditional backscattering applications include radar, fault detection
in fiber optics, Rutherford backscattering and X-ray backscattering
(e.g. full-body scanners) among others. What's common to all of these
is that the measured object (or fault) is located far away from the
wave source. From the point of view of the Rakesh-Uhlmann
\cite{RU1,RU2} techniques the classical backscattering problem in the
time-domain behaves as the point source problem with source at
infinity. This means that the problem \eqref{EQ1}--\eqref{EQ2} models
a situation where the wave source is close to the object under
investigation, for example in the order of a few wavelengths. Therefor
our results imply that backscattering experiments would give useful
information even when the object is close. For example one could
imagine using the backscattering of sound, radio or elastic waves to
find faults in an object of human scale.

Uniqueness for the inverse backscattering problem related to
\eqref{EQ1}--\eqref{EQ2} was shown by Rakesh and Uhlmann for an
admissible class of smooth potentials in \cite{RU2}. We shall show
stability for their method. In addition we will show that the direct
problem is well-posed in the sense of Hadamard, including all the
required norm estimates.

The question of well-posedness of the direct problem would seem
well-known to the experts at first sight. However this result is very
difficult to find in the literature for non-smooth potentials and with
explicit norm estimates. We hope that future research on the topic
finds the explicit proof convenient.

\medskip
The main motivation for this paper is the proof of the following
stability theorem. As in \cite{RU1,RU2} it applies to a class of
potentials whose differences are \emph{angularly controlled}.

\begin{theorem} \label{inverseThm}
  Let $\DOI = B(\bar0,1) \subset \R^3$ and fix positive a-priori
  parameters $S,\mathcal M < \infty$ and $h<1$. Then there are
  $\mathfrak C, \mathfrak D < \infty$ with the following properties:

  Let $q_1, q_2 \in \qspace$ with norm bounds $\qnorm{q_j} \leq
  \mathcal M$. Assume moreover that $\supp q_1$ and $\supp q_2$ are no
  closer than distance $h$ from $\partial\DOI$. If $q_1-q_2$ is
  angularly controlled with constant $S$, i.e.
  \begin{equation} \label{angularControl}
    \sum_{i<j} \int_{\abs{x}=r} \abs{\Omega_{ij}(q_1-q_2)(x)}^2
    d\sigma(x) \leq S^2 \int_{\abs{x}=r} \abs{(q_1-q_2)(x)}^2
    d\sigma(x)
  \end{equation}
  for any $0<r<1$ where $\Omega_{ij} = x_i\partial_j - x_j\partial_i$
  are the angular derivatives, then we have the following conditional
  stability estimate
  \begin{equation} \label{condStab}
    \norm{q_1-q_2}_{L^2(\{\abs{x}=r\})} \leq e^{\mathfrak C / r^4}
    \norm{U_1^a-U_2^a}
  \end{equation}
  for any given positive $r$. Here $U^a_1$ and $U^a_2$ are the unique
  solutions to the problem \eqref{EQ1}--\eqref{EQ2} given by Theorem
  \ref{directProblemOK} with $a\in\partial\DOI$, $q=q_1$, $q=q_2$, and
  \[
  \norm{U_1^a-U_2^a}^2 = \sup_{0<\tau<1} \int_{\abs{a}=1} \abs{
    \partial_\tau \big( \tau (U^a_1-U^a_2)(a,2\tau) \big) }^2
  d\sigma(a)
  \]
  is the backscattering measurement norm that we impose.

  A fortiori we get the logarithmic full-domain estimate
  \begin{equation} \label{condStabOrigin}
    \norm{q_1-q_2}_{L^2(\DOI)} \leq \mathfrak D \left(\ln
    \frac{1}{\norm{U_1^a-U_2^a}}\right)^{-1/4}
  \end{equation}
  when $\norm{U_1^a-U_2^a} < e^{-1}$ and $\norm{q_1-q_2}_{L^2(B)} \leq
  \mathfrak D \norm{U_1^a-U_2^a}$ otherwise.

  If instead of angular control for $q_1-q_2$ we assume the stronger
  condition of radial symmetry, we have
  \[
  \norm{q_1-q_2}_{L^2(\{\abs{x}=r\})} \leq \mathfrak C r^\alpha
  \norm{U_1^a-U_2^a}
  \]
  where $\alpha=\alpha(\mathcal M,h,\DOI)$, and this implies the full
  domain H\"older estimate
  \[
  \norm{q_1-q_2}_{L^2(\DOI)} \leq \mathfrak D
  \norm{U_1^a-U_2^a}^{\frac{1}{1+\alpha}}.
  \]
\end{theorem}

The proof of the above theorem is presented in Section
\ref{inverseSect} and is based on the innovative techniques from
\cite{RU2}. It starts with writing the data
$U_1^a(a,2\tau)-U_2^a(a,2\tau)$ as an integral involving $q_1-q_2$ and
solutions to \eqref{EQ1}--\eqref{EQ2}. The linear part of this
integral is the average of $q_1-q_2$ over spheres with centers on
$\partial B$. Proposition \ref{prop:differentiate} is key for
inverting the linearised problem and its perturbations. The inversion
formula to this, and to the corresponding linearized problem in
plane-wave inverse backscattering --- which is the Radon transform ---
is an ill-posed operator. Angular control and Gr\"onwall's inequality
give uniqueness and logarithmic stability to the linearized problem,
and also to the full nonlinear inverse problem.

From the point of view of applications the logarithmic stability seems
unpleasant. If we knew in advance that $q_1=q_2$ in a fixed
neighbourhood of the origin, then \eqref{condStab} would give us a
Lipschitz stability estimate $\norm{q_1-q_2}_{L^2(B)} \leq C
\norm{U_1^a-U_2^a}$. However it is not clear under which conditions
$q_1-q_2$ would stay angularly controlled if the origin was moved to
another location, e.g. outside of their supports. The method of this
paper and \cite{RU1,RU2} is centered around angular control so further
work should focus on understanding this condition. When the integrals
that use this condition are ignored, as happens when $q_1-q_2$ is
radially symmetric, we get H\"older stability.

It would be extremely surprising if H\"older stability was possible in
general. The fixed frequency multi-static inverse problem is known to
be exponentially ill-posed \cite{Mandache}. Counting dimensions, this
problem is overdetermined in $\R^3$ while the harder backscattering
problem is determined. However no formal inference can be made since
there is no known direct way of deducing the multi frequency (or
time-domain) backscattering data from the fixed frequency multi-static
data. Furter comments on this complex issue deserve a completely new
study.

\medskip Showing the well-posedness of the direct problem
\eqref{EQ1}--\eqref{EQ2} is a major effort. This has to be done for
two reasons. Firstly because the proof of Theorem \ref{inverseThm}
requires norm-estimates related to the solution $U^a$. These estimates
are lacking from the literature. Secondly, it makes sure that the
backscattering data $U^a(a,t)$ is smooth enough for the above theorem
to say anything meaningful.

\begin{theorem} \label{directProblemOK}
  Let $n\geq\qsmoothness$ and $\DOI=B(\bar0,1)$ be the unit disc in
  $\R^3$. Let $q\in\qspacen{n}$ and $a\in\partial\DOI$. Then the
  point source problem \eqref{EQ1}--\eqref{EQ2} has a unique solution
  $U^a$ in the set of distributions of order $n$. It is given by
  \begin{equation} \label{ansatz}
    U^a(x,t) = \frac{\delta(t-\abs{x-a})}{4\pi\abs{x-a}} +
    H(t-\abs{x-a}) r^a(x,t)
  \end{equation}
  where $r^a\in C^1(\R^3\times\R)$ and $\delta, H$ are the Dirac-delta
  distribution and Heaviside function on $\R$. For any $T>0$ and
  $\mathcal M \geq \qnorm{q}$ it has the norm estimate
  \begin{equation} \label{2ndTermEstimates}
    \norm{r^a}_{C^1(\R^3\times{[{0,T}]})} \leq C_{T,\mathcal M}.
  \end{equation}
  
  Moreover $U^a$ is $C^1$-smooth outside the light cone
  $t=\abs{x-a}$. In particular the map $(a,\tau) \mapsto U^a(a,2\tau)$
  is well-defined $\partial\DOI\times{({0,1})}\to\C$ and continuously
  differentiable in $\tau$. Furthermore
  \[
  \sup_{a\in\partial\DOI}\sup_{0<\tau<1}
  \abs{\partial_\tau^\beta(U_1^a-U_2^a)(a,2\tau)} \leq C_{\mathcal M}
  \qnorm{q_1-q_2}
  \]
  for solutions $U^a_j$ arising from two potentials $q_j$, $j=1,2$ and
  for any $\beta\in\{0,1\}$.
\end{theorem}

The proof of the above will be done by a \emph{progressive wave
  expansion}. This will lead us to a characteristic initial value
problem called the \emph{Goursat problem}. In \cite{RU2} this problem
was mentioned briefly with reference to \cite{Romanov1974}. Another
well-known source on the point source problem is
\cite{Friedlander}. The former studies the point source problem in low
regularity Sobolev spaces, which is not good enough since we need a
uniform $\partial_t$-estimate. The latter suffers from too much
generality and considers only $C^\infty$ smooth coefficients, without
any norm estimates. Neither reference mentions the Goursat problem by
name or defines it explicitly.

There are other sources, more focused on the Goursat problem. For
example \cite{Cagnac} is very detailed on the topic but seems to have
slightly larger smoothness requirements than we do. See also
\cite{BaleanPhD, Balean} for a very detailed analysis but their model
has a region removed from the middle of the characteristic
cone. Therefor we shall also prove well-posedness of the Goursat
problem.

\begin{theorem} \label{goursatWellPosed}
  For $n\in\N, n\geq5$ let $q\in C^n(\R^3)$ and $g\in C^{n+2}(\R^3)$
  with the norm bounds $\norm{q}_{C^n}\leq\mathcal{M}$ and
  $\norm{g}_{C^{n+2}}\leq\mathcal{N}$. Then there is a unique $C^1$
  solution $u$ to the problem
  \begin{alignat*}{2}
    (\partial_t^2 - \Delta - q) u &= 0, &\qquad& x\in\R^3, t >
    \abs{x}\\ u(x,t) &= g(x), &\qquad& x\in\R^3, t=\abs{x}.
  \end{alignat*}
  It is also in $C^s(\R^3\times\R)$ where
  $s=\lfloor\frac{n-2}{3}\rfloor$ and satisfies
  \[
  (\partial_t + \partial_r) u = \partial_r g, \qquad x\in\R^3,
  t=\abs{x}
  \]
  where $\partial_r = \frac{x}{\abs{x}} \cdot \nabla_x$.

  For any $T<\infty$ the solution has the norm estimate
  \[
  \norm{u}_{C^s(\R^3\times{[{0,T}]})} \leq C_{T, n, \mathcal{M}}
  \mathcal N.
  \]
  Finally, if $q_1,q_2\in C^n(\R^3)$ and $g_1,g_2\in C^{n+2}(\R^3)$
  then their corresponding solutions satisfy
  \[
  \norm{u_1-u_2}_{C^s(\R^3\times{[{0,T}]})} \leq C_{T, n, \mathcal{M},
    \mathcal N} \big( \norm{q_1-q_2}_{C^n(\R^3)} +
  \norm{g_1-g_2}_{C^{n+2}(\R^3)} \big).
  \]
\end{theorem}

\medskip
We will use the following notation for function spaces of continuous
functions.

\begin{definition}
  Let $s\in\N$ and $X\subset\R^d$ for some $d\in\Z_+$. The set
  $C^s(X)$ contains all $f\colon X\to\C$ that are $s$ times
  continuously differentiable. A subscript of $c$ as in $C^s_c(X)$
  indicates compact support in $X$.

  Given $s,\tau\in\N$ we denote by $C^{s,\tau}(\R^3\times\R)$ the
  space of continuous functions $f\colon \R^3\times\R\to\C$ for which
  $\partial_x^\alpha \partial_t^\beta f$ is continuous when
  $\alpha_1+\alpha_2+\alpha_3\leq s$ and $\beta\leq\tau$.

  For estimates,
  \begin{align*}
    &\norm{f}_{C^s(X)} = \sum_{\abs{\alpha}\leq s} \sup_{p\in X}
    \abs{\partial^\alpha f(p)}\\ &\norm{f}_{C^{s,\tau}(X)} =
    \sum_{\substack{\abs{\alpha}\leq s\\ \beta\leq\tau}}
    \sup_{(x,t)\in X} \abs{\partial_x^\alpha \partial_t^\beta f(x,t)}
  \end{align*}
  where $\alpha$ is a multi-index of appropriate dimension.
\end{definition}

A-priori no uniform bounds are required above. The solution to the
wave equation has finite speed of propagation so the qualitative
statements of our results stay true even for continuous but unbounded
functions.

\section{Goursat problem} \label{goursatSect}
The goal of this section is simple: prove the well-posedness of the
Goursat problem, including norm estimates of the solution with
dependence on the potential $q$ and Dirichlet data $g$ on the
characteristic cone. Before that we will show informally how the
point source problem is reduced to the \emph{Goursat problem}, or
\emph{characteristic initial-boundary value problem}. Lemma
\ref{goursat2PS} validates these informal calculations.

If $\delta, H \in \mathscr D'(\R)$ are the delta-distribution and
Heaviside function, then applying the operator $\partial_t^2 - \Delta
+ q$ to the ansatz
\begin{equation}
  U^a(x,t) = \frac{\delta(t-\abs{x-a})}{4\pi\abs{x-a}} +
  H(t-\abs{x-a}) r^a(x,t)
\end{equation}
gives
\begin{align*}
  &(\partial_t^2 - \Delta - q) U^a = (\partial_t^2-\Delta)
  \frac{\delta(t-\abs{x-a})}{4\pi\abs{x-a}} - \frac{q(x)
    \delta(t-\abs{x-a})}{4\pi \abs{x-a}} \\ & + \delta'(t-\abs{x-a})
  (r^a-r^a) + 2\frac{\delta(t-\abs{x-a})}{\abs{x-a}} \left(
  \abs{x-a}\partial_t r^a + r^a + (x-a) \cdot \nabla r^a \right) \\ &
  + H(t-\abs{x-a}) (\partial_t^2 - \Delta - q)r^a.
\end{align*}
Now $U^a$ will be a solution to \eqref{EQ1}--\eqref{EQ2} if
\begin{alignat*}{2}
  (\partial_t^2 - \Delta - q) r^a &= 0, &\qquad& x\in\R^3,
  t>\abs{x-a}, \\ \left( \abs{x-a}\partial_t + 1 + (x-a) \cdot \nabla
  \right) r^a &= \frac{q}{8\pi}, &\qquad& x\in\R^3, t=\abs{x-a}.
\end{alignat*}
However if $F(x) = \abs{x-a} r^a(x,\abs{x-a})$ then the chain rule
shows that
\begin{equation} \label{BC1}
  \frac{x-a}{\abs{x-a}} \cdot\nabla F = \left( \abs{x-a}\partial_t + 1
  + (x-a) \cdot \nabla \right) r^a(x,\abs{x-a}) = \frac{q(x)}{8\pi}
\end{equation}
and solving for $F$ gives
\begin{equation} \label{BC2}
  r^a(x,\abs{x}) = \frac{1}{8\pi} \int_0^1 q(a+s(x-a)) ds.
\end{equation}

Proving the converse requires more assumptions, so we will skip it
now. Instead we shall show that the Goursat problem
\begin{alignat}{2}
  (\partial_t^2 - \Delta - q) r^a &= 0, &\qquad& x\in\R^3,
  t>\abs{x-a}, \label{EQ1Goursat} \\ r^a &= g, &\qquad& x\in\R^3,
  t=\abs{x-a} \label{EQ2Goursat}
\end{alignat}
has a unique solution in $C^1$ for any $q$ and $g$ smooth enough, and
that this solution also satisfies the boundary condition \eqref{BC1}
when $g$ is chosen from \eqref{BC2}. Natural smoothness conditions are
$q\in C^n$ and $g\in C^{n+2}$.

\begin{definition} \label{gammaDef}
  For $k\in\Z$ define the function $\R^3\times\R\to\R$
  \[
  \gamma^k(x,t) = \begin{cases} \frac{(t^2-\abs{x}^2)^k}{k!},
    &k\in\N\\ 0, &k<0 \end{cases}.
  \]
\end{definition}

\begin{lemma} \label{progressiveWave}
  For $n\in\N$ let $q\in C^n(\R^3)$ and $g\in C^{n+2}(\R^3)$. Let
  $m\leq\lfloor\frac{n}{2}\rfloor+1$ be an integer. Then define
  $v\colon \R^3\times\R \to \C$ by
  \[
  v(x,t) = \sum_{k=0}^m a_k(x) \gamma^k(x,t)
  \]
  where the functions $a_k$ are defined as
  \begin{align}
    a_0(x) &= g(x), \qquad \R^3, \label{A0def}\\ a_{k+1}(x) &=
    \frac{1}{4} \int_0^1 s^{k+1} \big((q+\Delta) a_k\big) (xs) ds,
    \qquad \R^3. \label{Akdef}
  \end{align}
  Then $a_k \in C^{n+2-2k}(\R^3)$. They have the norm estimate
  \[
  \norm{a_k}_{C^{n+2-2k}(\R^3)} \leq \left(
  \frac{1+\norm{q}_{C^n(\R^3)}}{4} \right)^k \norm{g}_{C^{n+2}(\R^3)}.
  \]
  If $q_1,q_2\in C^n(\R^3)$ and $g_1,g_2\in C^{n+2}(\R^3)$ then for
  the corresponding sequences $a_{k1}$ and $a_{k2}$ we have
  \[
  \norm{a_{k1}-a_{k2}}_{C^{n+2-2k}} \leq \big(1+\mathcal M\big)^k
  \norm{g_1-g_2}_{C^{n+2}} + k\big(1+\mathcal M\big)^{k-1} \mathcal N
  \norm{q_1-q_2}_{C^n}
  \]
  whenever $\mathcal M\geq\norm{q_j}_{C^n}$ and $\mathcal
  N\geq\norm{g_j}_{C^{n+2}}$· Moreover
  \begin{alignat*}{2}
    (\partial_t^2 - \Delta - q) v &= -(q+\Delta) a_m \gamma^m,
    &\qquad& x\in\R^3, t\in\R, \\ v(x,t) &= g(x), &\qquad& x\in\R^3,
    t=\pm\abs{x}.
  \end{alignat*}
\end{lemma}
\begin{proof}
  Let us start by showing the norm estimates. Obviously $a_0 \in
  C^{n+2}(\R^3)$ with estimate $\norm{a_0}_{C^{n+2}} =
  \norm{g}_{C^{n+2}}$ and $a_{0j}\to a_0$ in norm.  Assume that
  $a_k\in C^{n+2-2k}$. Then $qa_k$ has smoothness $\min(n, n+2-2k)$,
  and $\Delta a_k$ has smoothness $n-2k$. Hence $a_{k+1}$ has
  smoothness $n-2k$ at worst, with norm estimate
  \[
  \norm{a_{k+1}}_{C^{n-2k}} \leq \frac{1}{4}(1+\norm{q}_{C^n})
  \norm{a_k}_{C^{n+2-2k}}
  \]
  whose coefficient could be improved by taking into account the value
  of the integral $\int_0^1 s^{k+1} ds$. The norm estimate for a
  general $k$ is
  \[
  \norm{a_k}_{C^{n+2-2k}} \leq \left( \frac{1+\norm{q}_{C^n}}{4}
  \right)^k \norm{g}_{C^{n+2}}
  \]
  by induction.

  For the difference we note that
  \[
  \big(a_{(k+1)1}-a_{(k+1)2}\big)(x) = \frac{1}{4} \int_0^1 s^{k+1}
  \big( (q_1+\Delta)a_{k1} - (q_2+\Delta)a_{k2} \big)(xs) ds
  \]
  and thus
  \begin{align*}
    &\norm{a_{(k+1)1}-a_{(k+1)2}}_{C^{n-2k}} \\ &\qquad\leq
    (1+\norm{q_1}_C^n) \norm{a_{k1}-a_{k2}}_{C^{n+2-2k}} +
    \norm{q_1-q_2}_{C^n} \norm{a_{k2}}_{C^{n-2k}} \\ &\qquad\leq
    \big(1+\mathcal M\big) \norm{a_{k1}-a_{k2}}_{C^{n+2-2k}} +
    \big(1+\mathcal M\big)^k \mathcal N \norm{q_1-q_2}_{C^n}
  \end{align*}
  in terms of the a-priori bounds. The norm estimate for the
  difference is now a simple induction.

  The claim $(\partial_t^2 - \Delta - q) v = -(q+\Delta) a_m \gamma^m$
  follows from noting that $a_0 = g$, $4x\cdot\nabla a_{k+1} +
  4(2+k)a_{k+1} - (q+\Delta)a_k = 0$, and $\partial_t \gamma^k = 2t
  \gamma^{k-1}$, $\nabla \gamma^k = -2x \gamma^{k-1}$, and then
  finally applying $\partial_t^2 - \Delta - q$ to the definition of
  $v$.
\end{proof}

\begin{lemma} \label{initialProblem}
  Let $n,\tau\in\N$, $q \in C^n(\R^3)$ and $F \in
  C^{n,\tau}(\R^3\times\R)$. Assume that $F(x,t)=0$ when $t<\abs{x}$,
  and consider the problem
  \begin{alignat}{2}
    (\partial_t^2 - \Delta - q) w &= F, &\qquad& x\in\R^3,
    t\in\R, \label{EQ1initial}\\ w &= 0, &\qquad& x\in\R^3,
    t<0. \label{EQ2initial}
  \end{alignat}
  It has a solution $w\in C^{n,\tau}(\R^3\times\R)$ which moreover
  vanishes on $t<\abs{x}$. Given $T<\infty$ and $\mathcal
  M\geq\norm{q}_{C^n(\R^3)}$ it satisfies
  \[
  \norm{w}_{C^{n,\tau}(\R^3\times{[{0,T}]})} \leq C_{T,n,\mathcal M}
  \norm{F}_{C^{s,\tau}(\R^3\times{[{0,T}]})}
  \]
  where
  \[
  C_{T,n,\mathcal M} = C_{n,\tau} \sum_{m=0}^\infty \frac{C_n^m
    \mathcal M^m T^{2(m+1)}}{4^{m+1} (m+1)! (m+2)!} < \infty
  \]
  and $C_{n,\tau}$ and $C_n$ are finite and depend only on the
  parameters in their indices.

  Finally, given such $q_1,q_2$ and $F_1,F_2$ let $w_1,w_2$ be the
  corresponding solutions. With the a-priori bounds
  $\norm{q_j}_{C^n(\R^3)}\leq\mathcal M$ and
  $\norm{F_j}_{C^{n,\tau}(\R^3\times{[{0,T}]})}\leq\mathcal N$ we have
  \[
  \norm{w_1-w_2}_{C^{n,\tau}(\R^3\times{[{0,T}]})} \leq
  C_{T,n,\mathcal M,\mathcal N} \big(
  \norm{F_1-F_2}_{C^{n,\tau}(\R^3\times{[{0,T}]})} +
  \norm{q_1-q_2}_{C^n(\R^3)} \big)
  \]
  where $C_{T,n,\mathcal M, \mathcal N}$ is finite and depends only on
  the parameters in its indices.
\end{lemma}
\begin{proof}
  Consider the operator
  \[
  K f(x,t) = \int_{\R^3} \frac{f(x-y,t-\abs{y})}{4\pi\abs{y}} dy
  \]
  giving $(\partial_t^2 - \Delta) K f = f$ for compactly supported
  distributions $f \in \mathscr E'(\R^3\times\R)$ and $K f(x,t) = 0$
  for $t<\inf_t \supp f$. This is also true for $f$ supported on
  $\abs{x} \leq t$ (see Theorem 4.1.2 in \cite{Friedlander}) and then
  the integration area becomes $\abs{x-y}+\abs{y}\leq t$. By Lemma
  \ref{ellipticIntegral}
  \[
  \abs{\partial_x^\alpha \partial_t^\beta Kf(x,t)} \leq \begin{cases}
    \sup_{\R^3 \times {]{{-\infty},t}[}} \abs{\partial_x^\alpha
      \partial_t^\beta f} \frac{t^2-\abs{x}^2}{8},& t>\abs{x}\\ 0,&
    t\leq\abs{x} \end{cases}
  \]
  when $\partial_x^\alpha \partial_t^\beta f$ is a continuous
  function. In essence $Kf$ has the same smoothness properties as $f$.

  The equation $(\partial_t^2 - \Delta - q) w = F$ with $w=0$ for
  negative time is equivalent to $w = K F + K(qw)$. Set $w_0(x,t) =
  KF(x,t)$ and $w_{m+1} = K(q w_m)$, and we will build the final
  solutions as
  \[
  w = \sum_{m=0}^\infty w_m.
  \]
  We see immediately by the properties of $K$ that $w_m\in
  C^{n,\tau}(\R^3\times\R)$ for all $m$ and that they vanish on
  $t<\abs{x}$. Moreover
  \[
  \abs{\partial_x^\alpha \partial_t^\beta w_0(x,t)} \leq
  \sup_{\R^3\times{[{0,t}]}} \abs{\partial_x^\alpha \partial_t^\beta
    F} \frac{t^2-\abs{x}^2}{8}
  \]
  when $t>\abs{x}$ and $\alpha_1+\alpha_2+\alpha_3\leq n$,
  $\beta\leq\tau$.

  Let us prove the claim by induction. Assume that for any
  $\alpha_1+\alpha_2+\alpha_3\leq n$ and $\beta\leq\tau$ we have
  \begin{equation} \label{inductionAssumption}
    \abs{\partial_x^\alpha \partial_t^\beta w_m(x,t)} \leq C_m
    \norm{q}_{C^n(\R^3)}^m \norm{F}_{C^{n,\tau}(\R^3\times{[{0,t}]})}
    (t^2-\abs{x}^2)^{m+1}
  \end{equation}
  for some $C_m$ which might depend on the other parameters. Then
  recall $w_m=0$ for $t<\abs{x}$ and the definition of $w_{m+1}$. We
  get
  \begin{align*}
    &\abs{\partial_x^\alpha \partial_t^\beta w_{m+1}(x,t)} = \abs{
      \int_{\R^3} \frac{\partial_x^\alpha \big(q(x-y) \partial_t^\beta
        w_m(x-y, t-\abs{y})\big)}{4\pi\abs{y}} dy } \\ &\qquad \leq
    C_m \sum_{\gamma\leq\alpha} {\alpha\choose\gamma}
    \norm{q}_{C^n(\R^3)}^{m+1}
    \norm{F}_{C^{n,\tau}(\R^3\times{[{0,t}]})} \\ &\qquad\quad \cdot
    \int_{\abs{x-y}+\abs{y}\leq t} \frac{\big((t-\abs{y})^2 -
      \abs{x-y}^2\big)^{m+1}}{4\pi\abs{y}} dy\\ &\qquad = \frac{C_m
      C_{s,n}}{4(m+2)(m+3)} \norm{q}_{C^n(\R^3)}^{m+1}
    \norm{F}_{C^{n,\tau}(\R^3\times{[{0,t}]})} (t^2-\abs{x}^2)^{m+2}
  \end{align*}
  where the last equality comes from Lemma \ref{ellipticIntegral}, and
  where
  \[
  C_n = \max_{\abs{\alpha}\leq n} \sum_{\gamma\leq\alpha}
  {\alpha\choose\gamma}.
  \]
  We also have $w_{m+1}(x,t)=0$ for $t<\abs{x}$. Hence we have the
  recursion formula $C_{m+1} = C_m C_n / (4(m+2)(m+3))$ and $C_0 =
  1/8$. This implies that \eqref{inductionAssumption} holds with
  \[
  C_m = \frac{C_n^m}{4^{m+1}(m+1)!(m+2)!}
  \]
  for $m=0,1,\ldots$.

  The series
  \[
  \sum_{m=0}^\infty \abs{\partial_x^\alpha \partial_t^\beta w_m(x,t)}
  \leq \sum_{m=0}^\infty \frac{C_n^m \norm{q}_{C^n(\R^3)}^m
    (t^2-\abs{x}^2)^{m+1}}{4^{m+1} (m+1)!  (m+2)!}
  \norm{F}_{C^{n,\tau}(\R^3\times{[{0,t}]})}
  \]
  converges uniformly for any $t, \abs{x}$ under a given bound, so the
  function $w$ is well defined. Note that the extension of
  $t^2-\abs{x}^2$ by zero to $t<\abs{x}$ is continuous. Hence
  $\partial_x^\alpha \partial_t^\beta w$ is continuous in
  $\R^3\times\R$ when $\alpha_1+\alpha_2+\alpha_3\leq n$ and
  $\beta\leq\tau$. Thus $w\in C^{n,\tau}(\R^3\times\R)$.

  The final claim, continuous dependence on $q$ and $F$, follows from
  the previous estimates. Namely, we note that $w_1$ and $w_2$ satisfy
  the assumptions of the source term $F$, and the difference $w_1-w_2$
  solves
  \[
  (\partial_t^2 - \Delta - q_1)(w_1-w_2) = F_1 - F_2 + (q_1-q_2)w_2
  \]
  with $w_1-w_2=0$ for $t<\abs{x}$. The
  $C^{n,\tau}(\R^3\times{[{0,T}]})$-norm of the right-hand side is
  bounded above by
  \[
  C_{T,n,\mathcal M} \big( \norm{F_1-F_2}_{C^{n,\tau}} +
  \norm{q_1-q_2}_{C^n} C_{T,n,\mathcal M} \norm{F_2}_{C^{n,\tau}}
  \big)
  \]
  and the claim follows from the a-priori bound on $F_2$.
\end{proof}

\begin{lemma} \label{uniqueGoursat}
  Let $u\colon \R^3\times\R\to\C$ be a $C^1$-function satisfying
  \begin{alignat*}{2}
    (\partial_t^2 - \Delta - q) u &= 0, &\qquad& x\in\R^3, t >
    \abs{x}\\ u(x,t) &= g(x), &\qquad& x\in\R^3, t=\abs{x}
  \end{alignat*}
  for some $q\in C^0(\R^3)$ and $g\in C^1(\R^3)$. If $g=0$ then $u=0$
  in $\abs{x}\leq t$.
\end{lemma}
\begin{proof}
  Define
  \[
  E(t) = \int_{\abs{x}\leq t} (\abs{\partial_t u}^2 + \abs{\nabla u}^2
  + \abs{u}^2) dx.
  \]
  We would like to differentiate $E$ with respect to time, however the
  lack of continuous second derivatives prevents us from doing that
  directly. Let $\varphi_\varepsilon$ be a mollifier and
  $u_\varepsilon = \varphi_\varepsilon \ast u$. Let $E_\varepsilon(t)
  = \int_{\abs{x}\leq t} ( \abs{\partial_t u_\varepsilon}^2 +
  \abs{\nabla u_\varepsilon}^2 + \abs{u_\varepsilon}^2 ) dx$. Then
  \begin{align*}
    &E_\varepsilon'(t) = \int_{\abs{x}=t} \big( \abs{\partial_t
      u_\varepsilon}^2 + \abs{\nabla u_\varepsilon}^2 +
    \abs{u_\varepsilon}^2 \big) d\sigma(x) + 2 \Re \int_{\abs{x}\leq
      t} \partial_t u_\varepsilon \cdot \overline{\partial_t^2
      u_\varepsilon} dx\\ &\qquad\quad + 2\Re \int_{\abs{x}\leq t}
    \nabla \partial_t u_\varepsilon \cdot \overline{\nabla
      u_\varepsilon} dx + 2\Re \int_{\abs{x}\leq t} \partial_t
    u_\varepsilon \overline{u_\varepsilon} dx.
  \end{align*}
  Integration by parts shows that the third term is equal to
  \[
  2\Re \int_{\abs{x}=t} \frac{x}{\abs{x}} \partial_t u_\varepsilon
  \cdot \overline{\nabla u_\varepsilon} d\sigma(x) - 2 \Re
  \int_{\abs{x}\leq t} \partial_t u_\varepsilon \overline{\Delta
    u_\varepsilon} dx.
  \]
  By combining both equations above and using $\partial_t^2
  u_\varepsilon - \Delta u_\varepsilon = \varphi_\varepsilon\ast(qu)$
  we get
  \begin{align*}
    &E_\varepsilon'(t) = \int_{\abs{x}=t} \left( \abs{
      \frac{x}{\abs{x}} \partial_t u_\varepsilon + \nabla
      u_\varepsilon}^2 + \abs{u_\varepsilon}^2 \right)
    d\sigma(x)\\ &\qquad\quad + 2\Re \int_{\abs{x}\leq t} \partial_t
    u_\varepsilon \overline{(u_\varepsilon +
      \varphi_\varepsilon\ast(qu))} dx.
  \end{align*}
  Integrate this with respect to time. Since $u_\varepsilon \to u$ in
  $C^1$ locally as $\varepsilon\to0$, we get
  \begin{align*}
    &E(t) = \int_0^t \int_{\abs{x}=s} \left( \abs{ \frac{x}{\abs{x}}
      \partial_s u + \nabla u}^2 + \abs{u}^2 \right) d\sigma(x)
    ds\\ &\qquad\quad + \int_0^t 2\Re \int_{\abs{x}\leq s}
    (1+\overline{q}) \partial_s u \overline{u} dx ds.
  \end{align*}

  Let us deal with the boundary integral next. Define $u_b(x) =
  u(x,\abs{x})$. Then calculus shows that $\nabla u_b(x) = (\nabla u +
  \frac{x}{\abs{x}} \partial_t u)(x,\abs{x})$ because $\nabla \abs{x}
  = x/\abs{x}$· On the other hand the boundary condition of $u$ shows
  that $u_b=g$. Thus the formula inside the parenthesis above is equal
  to $\abs{\nabla g}^2 + \abs{g}^2$.

  Note that $\int_0^t \int_{\abs{x}=s} f(x) dx ds = \int_{\abs{x}\leq
    t} f(x) dx$ for time-independent functions $f$. Then, since
  $2\Re(A\overline{B}) \leq \abs{A}^2 + \abs{B}^2$, we get
  \[
  E(t) \leq \int_{\abs{x}\leq t} \big( \abs{\nabla g}^2 + \abs{g}^2
  \big) dx + (1+\norm{q}_\infty) \int_0^t \int_{\abs{x}\leq s} \big(
  \abs{\partial_s u}^2 + \abs{u}^2 \big) dx ds.
  \]
  The last integral has the upper bound $\int_0^t E(s) ds$.
  Gr\"onwall's inequality, for example \mbox{Appendix~B.2.k} in
  \cite{Evans}, shows that $E(t)=0$ when $g=0$.
\end{proof}

\bigskip
We are now ready to prove the well-posedness of the Goursat problem in
the sense of Hadamard. Strictly speaking the same proof shows
existence in $C^0$ when $q\in C^2$, $g\in C^4$, but then we cannot
guarantee uniqueness or the boundary identity that's stated with
$\partial_t$ and $\partial_r$.

\begin{proof}[Proof of Theorem \ref{goursatWellPosed}]
  This is a consequence of the uniqueness of Lemma
  \ref{uniqueGoursat}, the progressive wave expansion of Lemma
  \ref{progressiveWave} and the initial value problem of Lemma
  \ref{initialProblem}. Let $m=\lfloor(n+1)/3\rfloor$, which has
  $m\geq2$ and $n\geq2m+1$, and set
  \begin{equation} \label{vFixed}
    v(x,t) = g(x)+ a_1(x) (t^2-\abs{x}^2)+ \cdot+ a_m(x) \gamma^m(x,t)
  \end{equation}
  for $(x,t)\in\R^3\times\R$, as in Lemma \ref{progressiveWave}. We
  have $\norm{a_k}_{C^{n+2-2k}}\leq C_n(1+\mathcal{M})^k\mathcal{N}$
  in $\R^3$. Then $v(x,\abs{x}) = g(x)$ but $(\partial_t^2-\Delta-q)v
  = -(q+\Delta)a_m \gamma^m$.

  Next let
  \begin{equation} \label{Ffixed}
    F(x,t) = \begin{cases} (q+\Delta)a_m(x) \gamma^m(x,t),
      &t>\abs{x}\\ 0, &t\leq\abs{x}
    \end{cases}
  \end{equation}
  be our source term for an initial value problem. We have
  $(q+\Delta)a_m\in C^{n-2m}(\R^3)$, but $\chi_{\{t>\abs{x}\}}
  \gamma^m$ is in $C^{m-1}(\R^3\times\R)$. Hence $F \in
  C^{n_0,\tau_0}(\R^3\times\R)$ using the notation of Lemma
  \ref{initialProblem} whenever $n_0+\tau_0\leq m-1$ and
  $n_0\leq\min(n-2m,m-1)=m-1$. In other words when $n_0+\tau_0\leq
  s$. Given $T>0$ the source has the estimate
  \[
  \norm{F}_{C^{n_0,\tau_0}(\R^3\times{[{0,T}]})} \leq C_{T, n,
    \mathcal{M}} \mathcal{N}.
  \]
  We can also write out the estimate for $v$ now that the smoothness
  indices are fixed. Note that $\gamma^k$ is infinitely smooth in
  $\R^3\times\R$, and $a_m$ has the worst smoothness among all the
  coefficient functions in \eqref{vFixed}. Thus
  \begin{equation} \label{vestim}
    \norm{v}_{C^{n_0,\tau_0}(\R^3\times{[{0,T}]})} \leq C_{T, n,
      \mathcal{M}} \mathcal{N}
  \end{equation}
  too since $n_0\leq m$ and $a_k$ is independent of $t$.

  Let $w$ solve $(\partial_t^2-\Delta-q)w = F$ in $\R^3\times\R$ with
  $w=0$ for $t<0$. Lemma \ref{initialProblem} shows that such a $w$
  exists in $C^{n_0,\tau_0}(\R^3\times\R)$ and it has support on
  $t\geq\abs{x}$. Given $T>0$ it has the norm estimate
  \begin{equation} \label{westim}
    \norm{w}_{C^{n_0,\tau_0}(\R^3\times{[{0,T}]})} \leq C_{T, n,
      \mathcal{M}} \mathcal N
  \end{equation}
  by the estimate on $F$.

  Since $s\geq1$ then $F\in C^{0,1}\cap C^{1,0}$ with support in
  $t\geq\abs{x}$. This implies that $\partial_t w$ and $\nabla_x w$
  are continuous. Since $w=0$ when $t<\abs{x}$ we see that
  $(\partial_t+\frac{x}{\abs{x}}\cdot\nabla_x)w=0$ for
  $t\leq\abs{x}$. Next consider $v$. We see that on $t=\abs{x}$
  \[
  \partial_t \gamma^k(x,t) = \begin{cases} 2t, &k=1,\\0,
    &k\neq1 \end{cases}
  \]
  and
  \[
  \nabla_x \gamma^k(x,t) = \begin{cases} -2x, &k=1,\\0,
    &k\neq1 \end{cases},
  \]
  so $\partial_t v = 2t a_1$ and $\nabla_x v = \nabla g - 2x a_1(x)$
  if $t=\abs{x}$. This implies that
  \[
  \left(\partial_t + \frac{x}{\abs{x}} \cdot \nabla_x\right) v =
  \frac{x}{\abs{x}}\cdot\nabla g(x)
  \]
  on $t=\abs{x}$.
 
  If we set $u=v+w$, then we see that $u(x,\abs{x}) = g(x)$ and
  $(\partial_t+\partial_r)u=\partial_rg$ on $t=r=\abs{x}$ because $w$
  is continuous in $\R^3\times\R$ and supported on
  $t\geq\abs{x}$. Moreover $u\in C^s$ since
  \[
  \norm{u}_{C^s(\R^3\times{[{0,T}]})} \leq C \sup_{n_0+\tau_0\leq s}
  \norm{u}_{C^{n_0,\tau_0}(\R^3\times{[{0,T}]})}
  \]
  and this gives us the required norm estimate from \eqref{vestim} and
  \eqref{westim}. Finally $(\partial_t^2-\Delta-q)u =
  (\partial_t^2-\Delta-q)v + F = 0$ on $t>\abs{x}$.

  The estimate for the difference of solutions $u_1-u_2$ to two
  Goursat problems follows from the corresponding estimate for
  $v_1-v_2$ of Lemma \ref{progressiveWave} and for $w_1-w_2$ of Lemma
  \ref{initialProblem}. After using the latter note that
  \[
  \norm{F_1-F_2}_{C^{n_0,\tau_0}} \leq C_{T,n} \big(1+\mathcal M\big)
  \norm{a_{m1}-a_{m2}}_{C^{n_0}} + \norm{q_1-q_2}
  \norm{a_{m2}}_{C^{n_0}}
  \]
  holds and thus can be estimated above by the norms of $q_1-q_2$ and
  $g_1-g_2$.
\end{proof}

\section{Well-posedness of the point source backscattering measurements}
Now that the Goursat problem has been taken care of we can focus on
the point source problem. We will show that given a $\qspace$
potential $q$ there is a unique solution to \eqref{EQ1}--\eqref{EQ2},
and we can define the associated backscattering measurements. Moreover
these measurements depend continuously on the potential, with linear
modulus of continuity.

\begin{lemma} \label{goursat2PS}
  Let $q\in C^0_c(\DOI)$ and $a\in\partial\DOI$. Let $r^a\in
  C^1(\R^3\times\R)$ solve the problem
  \begin{alignat*}{2}
    (\partial_t^2 - \Delta - q) r^a &= 0, &\qquad& x\in\R^3,
    t>\abs{x-a}, \\ \left( \abs{x-a} \partial_t + 1 + (x-a) \cdot
    \nabla \right) r^a &= \frac{q}{8\pi}, &\qquad& x\in\R^3,
    t=\abs{x-a}.
  \end{alignat*}
  Define
  \[
  U^a(x,t) = \frac{\delta(t-\abs{x-a})}{4\pi\abs{x-a}} +
  H(t-\abs{x-a}) r^a(x,t)
  \]
  where $\delta, H \in \mathscr D'(\R)$ are the delta-distribution and
  Heaviside function. Then $U^a$ is a solution to the point source
  problem \eqref{EQ1}--\eqref{EQ2}.
\end{lemma}
\begin{proof}
  Take the above form of $U^a$ as an ansatz and note that the first
  term is the Green's function for $\partial_t^2 - \Delta$
  \begin{equation} \label{Green}
    (\partial_t^2 - \Delta) \frac{\delta(t-\abs{x-a})}{4\pi\abs{x-a}}
    = \delta(x-a,t)
  \end{equation}
  by for example Theorem 4.1.1 in \cite{Friedlander}.

  Since the function $r^a$ in our ansatz is a-priori only $C^1$, we
  will use a smoothened delta-distribution and Heaviside function. For
  $\varepsilon>0$ let $\delta_\varepsilon\colon \R\to\R$ be smooth,
  supported in ${]{0,2\varepsilon}[}$, positive, and $\int
  \delta_\varepsilon =1$. Let $H_\varepsilon(t) = \int_{-\infty}^t
  \delta_\varepsilon(s) ds$. Then $\delta_\varepsilon$ converges to
  the delta-distribution as $\varepsilon\to0$ and $H_\varepsilon$ to
  the Heaviside function. Let our new ansatz be
  \[
  U_\varepsilon(x,t) =
  \frac{\delta_\varepsilon(t-\abs{x-a})}{4\pi\abs{x-a}} +
  H_\varepsilon(t-\abs{x-a}) r^a(x,t).
  \]

  Let's calculate the derivatives of the second term in the ansatz
  next. Note that $\nabla \cdot (x/\abs{x}) = 2/\abs{x}$ in 3D, and so
  setting $R = H_\varepsilon(t-\abs{x-a}) r^a(t,x)$ we have
  \begin{align*}
    \partial_t R &= \delta_\varepsilon(t-\abs{x-a}) r^a +
    H_\varepsilon(t-\abs{x-a}) \partial_t r^a\\ \partial_t^2 R &=
    \delta_\varepsilon'(t-\abs{x-a}) r^a + 2
    \delta_\varepsilon(t-\abs{x-a}) \partial_t r^a +
    H_\varepsilon(t-\abs{x-a}) \partial_t^2 r^a, \\ \nabla R &=
    \delta_\varepsilon(t-\abs{x-a}) \left( -\frac{x-a}{\abs{x-a}}
    \right) r^a + H_\varepsilon(t-\abs{x-a}) \nabla r^a\\ \Delta R &=
    \delta_\varepsilon'(t-\abs{x-a}) r^a -
    \delta_\varepsilon(t-\abs{x-a}) \frac{2r^a}{\abs{x-a}}
    \\ &\phantom{=} - \delta_\varepsilon(t-\abs{x-a}) 2
    \frac{x-a}{\abs{x-a}} \cdot \nabla r^a +
    H_\varepsilon(t-\abs{x-a}) \delta_\varepsilon r^a, \\ q R &=
    H_\varepsilon(t-\abs{x-a}) q r^a.
  \end{align*}
  Take all terms into account next. Then
  \begin{align*}
    &(\partial_t^2 - \Delta - q) U_\varepsilon = (\partial_t^2-\Delta)
    \frac{\delta_\varepsilon(t-\abs{x-a})}{4\pi\abs{x-a}} - \frac{q(x)
      \delta_\varepsilon(t-\abs{x-a})}{4\pi \abs{x-a}} \\ & +
    \delta_\varepsilon'(t-\abs{x-a}) (r^a-r^a) +
    2\frac{\delta_\varepsilon(t-\abs{x-a})}{\abs{x-a}} \left(
    \abs{x-a}\partial_t r^a + r^a + (x-a) \cdot \nabla r^a \right)
    \\ & + H_\varepsilon(t-\abs{x-a}) (\partial_t^2 - \Delta - q)r^a.
  \end{align*}
  As $\varepsilon\to0$ the first term above converges to
  $\delta(x-a,t)$ in the space of distributions. The terms with
  coefficients $\delta_\varepsilon'$ and $H_\varepsilon$ vanish. The
  former trivially, and the latter because our choice of
  $\delta_\varepsilon$ makes sure that $\supp H_\varepsilon \subset
  \R_+$. In other words
  \begin{align*}
    &\lim_{\varepsilon\to0} (\partial_t^2 - \Delta - q)U_\varepsilon -
    \delta(x-a,t) \\ &\qquad = \lim_{\varepsilon\to0}
    2\frac{\delta_\varepsilon(t-\abs{x-a})}{\abs{x-a}} \left(
    \abs{x-a}\partial_t r^a + r^a + (x-a)\cdot\nabla r^a -
    \frac{q(x)}{8\pi} \right)
  \end{align*}
  in $\mathscr D'(\R^3\times\R)$.

  Denote by $f(x,t)$ the continuous function in parenthesis above. Let
  $\varphi\in C^\infty_c(\R^3\times\R)$ be a test function. Then in
  the support of $\varphi$ for every $\mu>0$ there is $\delta>0$ such
  that $\abs{f(x,t)}<\mu$ if $\abs{t-\abs{x-a}}<\delta$. Let
  $2\varepsilon<\delta$. Then
  \[
  \abs{\int_{\R^3\times\R}
    \frac{\delta_\varepsilon(t-\abs{x-a})}{\abs{x-a}} f(x,t)
    \varphi(x,t) dx dt} \leq \mu \norm{\varphi}_\infty \int_{\supp
    \varphi} \frac{\delta_\varepsilon(t-\abs{x-a})}{\abs{x-a}} dx dt
  \]
  and by integrating the $t$-variable first we get the upper bound
  \[
  \ldots \leq \mu \norm{\varphi}_\infty \int_{B(a,R_\varphi)}
  \frac{dx}{\abs{x-a}} = C_\varphi \mu.
  \]
  In other words the remaining term in the expansion for
  $(\partial_t^2-\Delta-q)U_\varepsilon$ tends to zero in the
  distribution sense. Hence
  \[
  (\partial_t^2 - \Delta - q) U_\varepsilon \to \delta(x-a,t)
  \]
  in $\mathscr D'(\R^3\times\R)$. Also, since $\supp
  \delta_\varepsilon \subset \R_+$, it also satisfies the initial
  condition $U_\varepsilon=0$ for $t<0$. Finally, it is easy to see
  that $U_\varepsilon \to U^a$. Hence the latter is a solution to
  \eqref{EQ1}--\eqref{EQ2}.
\end{proof}

\begin{lemma} \label{PSuniqueness}
  For $n\in\N$ let $q\in C^n(\R^3)$ and let $U$ be a distribution of
  order $n$ on $\R^3\times\R$ such that $U=0$ on $t<0$. If
  $(\partial_t^2 - \Delta - q)U=0$ then $U=0$.
\end{lemma}
\begin{proof}
  Let $\varphi\in C^\infty_c(\R^3\times\R)$ be arbitrary. There is
  $x_0\in\R^3$ and $t_0\in\R$ such that $\varphi(x,t)=0$ in
  $\abs{x-x_0}>t_0-t$, i.e. outside a past light cone. Write $y=x-x_0$
  and $s=t_0-t$, and define
  \[
  Q(y) = q(y+x_0), \qquad F(y,s) = \varphi(y+x_0,t_0-s).
  \]
  Then $Q\in C^n(\R^3)$, $F\in C^\infty_c(\R^3\times\R)$ and
  $F(y,s)=0$ when $s<\abs{y}$. Lemma \ref{initialProblem} gives the
  existence of $w\in C^n(\R^3\times\R)$ which vanishes on $s<\abs{y}$
  and satisfies $(\partial_s^2 - \Delta - Q)w = F$.

  Let
  \[
  \psi(x,t) = w(x-x_0,t_0-t).
  \]
  Then $\psi(x,t)=0$ if $\abs{x-x_0}>t_0-t$. Since $U=0$ for $t<0$,
  the intersection of the supports of $\psi$ and $U$ is a compact
  set. Since $U$ is of order $n$ and $\psi$ is in $C^n$ their
  distribution pairing $\langle U, \psi \rangle$ is well defined. Now
  \begin{align*}
    &\langle (\partial_t^2 - \Delta - q)U, \psi \rangle = \langle U,
    (\partial_t^2 - \Delta - q)\psi \rangle \\ &\qquad = \langle
    \tilde U, (\partial_s^2 - \Delta - Q)w \rangle = \langle \tilde U,
    F \rangle = \langle U, \varphi \rangle
  \end{align*}
  where $\tilde U$ is the distribution $U$ in the
  $(y,s)$-coordinates. Since $U$ is in the kernel of the differential
  operator and $\varphi$ is an arbitrary test function, we have $U=0$.
\end{proof}

\bigskip
\begin{proof}[Proof of Theorem \ref{directProblemOK}]
  Uniqueness follows directly from Lemma \ref{PSuniqueness}. We shall
  build a solution $r^a$ to the Goursat-type problem of Lemma
  \ref{goursat2PS}. We switch boundary conditions as was done at the
  beginning of Section \ref{goursatSect}. Define
  \[
  g(x) = \frac{1}{8\pi} \int_0^1 q\big( a + s(x-a) \big) ds
  \]
  and note that $q\in C^n(\R^3)$, $g\in C^{n+2}(\R^3)$ for $n=5$. The
  well-posedness of the Goursat problem (Theorem
  \ref{goursatWellPosed}) gives a unique $C^1$ solution to
  \begin{alignat*}{2}
    (\partial_t^2 - \Delta - q) r^a &= 0, &\qquad& x\in\R^3,
    t>\abs{x-a},\\ r^a &= g, &\qquad& x\in\R^3, t=\abs{x-a}.
  \end{alignat*}
  It has the required norm estimate for any $T>0$ and in addition it
  satisfies
  \[
  (\partial_t + \partial_r) r^a = \partial_r g
  \]
  on $t=\abs{x-a}$. Here $r=\abs{x-a}$ and furthermore we denote
  $\theta=(x-a)/\abs{x-a}$. If in the definition of $g$ we switch
  integration variables to $s'=rs$ then
  \[
  \partial_r g = -\frac{1}{r} g + \frac{q}{8\pi r}
  \]
  which is well-defined because $q=0$ in a neighbourhood of
  $a$. Recalling that $\partial_r = \theta\cdot\nabla_x$ we see that
  in fact
  \[
  (\abs{x-a}\partial_t + 1 + (x-a)\cdot\nabla_x) r^a = \frac{q}{8\pi}
  \]
  on the boundary $t=\abs{x-a}$. Hence Lemma \ref{goursat2PS} shows
  that $U^a$ is a solution to the point source problem.

  The unperturbed Green's function is supported only on
  $t=\abs{x-a}$. On $t<\abs{x-a}$ the solution vanishes. On
  $t>\abs{x-a}$ it is equal to $r^a$ which is $C^1$. In this topology,
  it depends continuously on $a$ because the Goursat problem depends
  continuously on the potential and characteristic boundary
  data. Hence $U(a,2\tau)$ is well-defined for $\tau>0$ and
  continuously differentiable in $\tau$.

  Let two potentials $q_1$ and $q_2$ and their associated solutions
  $r_1^a$, $r_2^a$ to the Goursat problem be given. For any
  $a\in\partial\DOI$ and $\beta\in\{0,1\}$ Theorem
  \ref{goursatWellPosed} shows the norm estimate
  \[
  \sup_{x\in\R^3}\sup_{0<\tau<1} \abs{\partial_\tau^\beta
    (r_1^a-r_2^a)(x,2\tau)} \leq C_{\mathcal M} \qnorm{q_1-q_2}
  \]
  because $\norm{g_1-g_2}_{C^7(\R^3)}\leq\norm{q_1-q_2}_{C^7(\R^3)}$
  and the norms involved are invariant under translations. Letting
  $x=a$ and then taking the supremum over $a$ proves the claim because
  $U_1^a-U_2^a = r_1^a-r_2^a$ at $(x,t)=(a,2\tau)$.
\end{proof}

\section{Stability of the inverse problem} \label{inverseSect}
Now that the direct problem has been shown to be well-defined,
including the estimates for the point source backscattering
measurements, we can consider the inverse problem. The first step is
to write a boundary identity. The following is proven in \cite{RU2}
for $C^\infty$-smooth potentials, but it works verbatim in our case
too.
\begin{proposition} \label{prop:bndry2inside}
  Let $\DOI=B(\bar0,1)$ be the unit ball in $\R^3$ and $q_1, q_2
  \in \qspace$. Let $a \in \partial\DOI$ and let $U^a_1$ and $U^a_2$
  be given by Theorem \ref{directProblemOK} for $q=q_j$, $j=1,2$. Then
  \begin{equation} \label{bndryIdentity}
    \begin{split}
      U^a_1(a,2\tau) - U^a_2(a,2\tau) = &\frac{1}{32\pi^2\tau^2}
      \int_{\abs{x-a}=\tau} (q_1-q_2)(x) d\sigma(x) \\ &+
      \int_{\abs{x-a}\leq\tau} (q_1-q_2)(x) k(x,\tau,a) d x
    \end{split}
  \end{equation}
  with
  \[
  k(x,\tau,a) = \frac{(r^a_1+r^a_2)(x,2\tau-\abs{x-a})}{4\pi\abs{x-a}}
  + \int_{\abs{x-a}}^{2\tau-\abs{x-a}} r^a_1(x,2\tau-t) r^a_2(x,t) dt
  \]
  if $\abs{x-a}\leq\tau$.

  If we have moreover $\qnorm{q_j} \leq \mathcal M < \infty$ then
  \begin{align}
    &\sup_{h\leq\tau\leq1} \sup_{\abs{a}=1} \int_{\abs{x-a}=\tau}
    \abs{k(x,\tau,a)}^2 d\sigma(x) \leq C_{\mathcal M,h,\DOI} <
    \infty,
    \label{kEst1} \\ &\sup_{h\leq\tau\leq1} \sup_{\abs{a}=1}
    \int_{h\leq\abs{x-a}\leq\tau} \abs{\partial_\tau (\tau
      k(x,\tau,a))}^2 d\sigma(x) \leq C_{\mathcal M,h,\DOI} <
    \infty \label{kEst2}
  \end{align}
  for any $h>0$. Note that $k(x,\tau,a)$ is singular at $x=a$.
\end{proposition}

\begin{proof}
  We shall skip the proof of the identities as they have been proved
  in Section 3.2 of \cite{RU2}. It is a matter of calculating
  \[
  \int_{-\infty}^\infty \int_{\R^n} (q_1-q_2)(x) U^a_2(x,t)
  U^a_1(x,2\tau-t) dx dt
  \]
  on one hand by integrating by parts, and on the other hand by using
  the expansion \eqref{ansatz}. The estimates for $k$ follow directly
  from \eqref{2ndTermEstimates}.
\end{proof}

Our next step is an integral identity related to the first term in
\eqref{bndryIdentity}. The proof for the estimate for $E(a,\tau)$ can
be dug from the proofs in \cite{RU2}. We prove it again here, both for
clarity, since this estimate might be of interest on its own, and for
having an explicit form for the constant in front of the sum.

\begin{proposition} \label{prop:differentiate}
  Let $Q\in C^1_c(\DOI)$ with $\DOI$ the unit disc in $\R^3$. Then for
  all $a\in\partial\DOI$ and $0<\tau<\abs{a}$ we have
  \begin{equation} \label{differentiate}
    \partial_\tau \left( \frac{\tau}{4\pi\tau^2} \int_{\abs{x-a}=\tau}
    Q(s) d\sigma(x) \right) = \frac{1-\tau}{2}Q\big((1-\tau)a\big) +
    E(a,\tau)
  \end{equation}
  where
  \[
  \abs{E(a,\tau)}^2 \leq \frac{3}{\pi(1-\tau)} \sum_{i<j}
  \int_{\abs{x-a}=\tau}
  \frac{\abs{\Omega_{ij}Q(x)}^2}{\sqrt{\abs{x}-(1-\tau)}} d\sigma(x).
  \]
  Here the $\Omega_{ij}$ are the angular derivatives $x_i\partial_j -
  x_j\partial_i$ depicted as vector fields in
  Figure~\ref{fig:vectorfields}.
  \begin{figure}
  \begin{center}
    \includegraphics{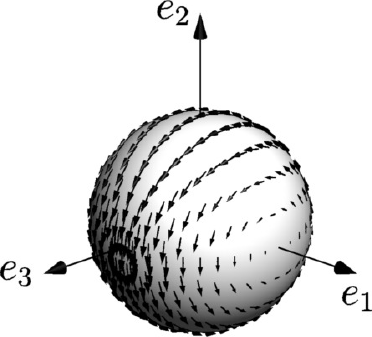} \includegraphics{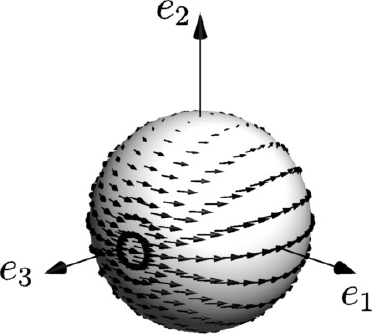}
    \includegraphics{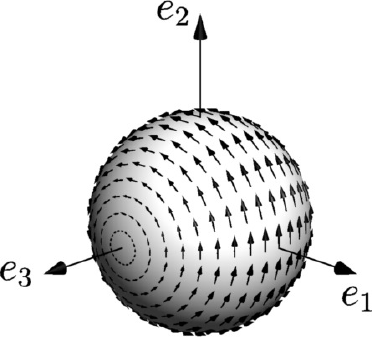}
    \caption{Angular derivatives $\Omega_{ij}$}
    \label{fig:vectorfields}
  \end{center}
  \end{figure}
\end{proposition}

\begin{proof}
  We may prove the proposition for $Q\in C^\infty_c(B)$ and then get
  the claim by approximating. Test functions are dense in
  $C^1_c(\DOI)$ and $\sup\abs{f} + \sup\abs{\nabla f} \leq C
  \norm{f}_{C^1}$. By Proposition 2.1 in \cite{RU2}
  \[
  \partial_\tau \left( \frac{\tau}{4\pi\tau^2} \int_{\abs{x-a}=\tau}
  Q(s) d\sigma(x) \right) = \frac{1-\tau}{2}Q\big((1-\tau)a\big) +
  \frac{1}{4\pi} \int_{\abs{x-a}=\tau} \frac{\alpha\cdot\nabla
    Q(x)}{\sin\phi} d\sigma(x),
  \]
  where $\alpha = \alpha(a,x)$ is a unit vector orthogonal to $x$ and
  $\phi$ is the angle at the origin between $x$ and $a$.
  \begin{figure}
  \begin{center}
    \includegraphics{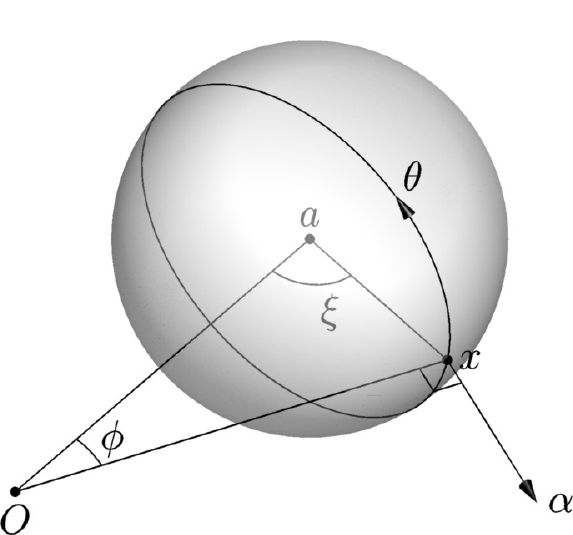}
    \caption{Reparametrization of $\abs{x-a}=\tau$}
    \label{fig:variables}
  \end{center}
  \end{figure}

  Let $T_{ij} = x_ie_j - x_je_i$ so $\Omega_{ij} = T_{ij} \cdot
  \nabla$. Then for any vector $v$ we have
  \[
  v = \sum_{i<j} \left( v \cdot \frac{T_{ij}}{\abs{x}} \right)
  \frac{T_{ij}}{\abs{x}} + \left(v\cdot \frac{x}{\abs{x}}\right)
  \frac{x}{\abs{x}}.
  \]
  On $\abs{x-a}=\tau$ set $v := \alpha$ and then take the dot product
  with $\nabla Q(x)$. We get
  \[
  \abs{x}^2 \alpha\cdot\nabla Q(x) = \sum_{i<j} (\alpha\cdot
  T_{ij})(T_{ij}\cdot\nabla Q)(x) = \sum_{i<j} (\alpha\cdot T_{ij})
  \Omega_{ij}Q (x)
  \]
  since $x\perp\alpha$. By the Cauchy-Schwarz inequality
  \[
  \abs{\alpha\cdot\nabla Q(x)} \leq \frac{\abs{a}}{\abs{x}} \sum_{i<j}
  \abs{\Omega_{ij}Q(x)}
  \]
  since $\abs{T_{ij}} \leq \abs{x}$. This implies
  \[
  \abs{E(a,\tau)} \leq \frac{\abs{a}}{4\pi} \sum_{i<j}
  \int_{\abs{x-a}=\tau}
  \frac{\abs{\Omega_{ij}Q(x)}}{\abs{x}\abs{\sin\phi}} d\sigma(x).
  \]

  The law of cosines gives us $2\abs{a}\abs{x}\cos\phi = \abs{a}^2 +
  \abs{x}^2 - \tau^2$. Solve for $\cos\phi$ to get $\sin\phi =
  \pm\sqrt{1-\cos^2\phi}$ and hence
  \begin{align*}
    &\frac{1}{\abs{\sin\phi}} =
    \frac{2\abs{a}\abs{x}}{\sqrt{4\abs{a}^2\abs{x}^2 - (\abs{a}^2 +
        \abs{x}^2 - \tau^2)^2}} \\ &\qquad =
    \frac{2\abs{a}\abs{x}}{\sqrt{(\abs{x}-\tau+\abs{a})
        (\abs{x}+\tau-\abs{a}) (\tau+\abs{a}-\abs{x})
        (\tau+\abs{a}+\abs{x})}}.
  \end{align*}
  But note that by assumption $\abs{a} > \tau > 0$ and $\abs{a} >
  \abs{x}$ for all $x\in\DOI$. Hence 
  \[
  \frac{1}{\abs{\sin\phi}} \leq
  \frac{2\abs{a}\abs{x}}{\sqrt{\abs{a}-\tau}
    \sqrt{\abs{x}-(\abs{a}-\tau)} \sqrt{\tau} \sqrt{\abs{a}}}.
  \]
  and we can continue with
  \[
  \abs{E(a,\tau)} \leq \frac{\abs{a}^2}{2\pi \sqrt{\tau\abs{a}}
    \sqrt{\abs{a}-\tau}} \sum_{i<j} \int_{\abs{x-a}=\tau}
  \frac{\abs{\Omega_{ij}Q(x)}}{\sqrt{\abs{x}-(\abs{a}-\tau)}}
  d\sigma(x).
  \]

  Finally, use the Cauchy-Schwarz inequality twice: once for
  $(\sum_{i<j}f_{ij})^2 \leq 3 \sum_{i<j} f_{ij}^2$ and a second time
  for the product of the two function
  $\abs{\Omega_{ij}Q(x)}/(\abs{x}-(\abs{a}-\tau))^{1/4}$ and
  $(\abs{x}-(\abs{a}-\tau))^{-1/4}$. It gives
  \[
  \abs{E(a,\tau)}^2 \leq \frac{3\abs{a}^3 I(a,\tau)}{4\pi^2 \tau
    (\abs{a}-\tau)} \sum_{i<j} \int_{\abs{x-a}=\tau}
  \frac{\abs{\Omega_{ij}Q(x)}^2}{\sqrt{\abs{x}-(\abs{a}-\tau)}}
  d\sigma(x)
  \]
  where $I(a,\tau) = \int_{\abs{x-a}=\tau, \abs{x}\leq\abs{a}}
  d\sigma(x) / \sqrt{\abs{x}-(\abs{a}-\tau)}$.
  
  Parametrize the sphere $\abs{a-x}=\tau$ by $\rho=\abs{x}$ and the
  azimuth $\theta\in{[{0,2\pi}]}$ to calculate $I(a,\tau)$. The latter
  variable gives the inclination of the plane $aOx$ with respect to a
  fixed reference plane passing through $O$ and $a$. See Figure
  \ref{fig:variables}. We also introduce the polar angle $\xi$.  Using
  the standard spherical coordinates $\xi$, $\theta$ we have
  \[
  d\sigma(x) = \tau^2 \sin\xi d\xi d\theta = \tau^2 \sin\xi
  \frac{d\xi}{d\rho} d\rho d\theta.
  \]
  By the law of cosines $\abs{a}^2 + \tau^2 - 2\abs{a}\tau\cos\xi =
  \rho^2$. Solve for $\cos\xi$ and differentiate this with respect to
  the variable $\rho$. Note that $a,\tau$ are constants, but
  $\xi=\xi(\rho)$. We get
  \[
  -\sin\xi\frac{d\xi}{d\rho} = \frac{d}{d\rho}\cos\xi = -
  \frac{\rho}{\abs{a}\tau}
  \]
  which implies that $d\sigma(x) = \tau \abs{a}^{-1} \rho d\rho
  d\theta$.

  Thus, since $Q$ vanishes outside $\DOI$, we have
  \[
  I(a,\tau) = \int_0^{2\pi} \int_{\abs{a}-\tau}^{\abs{a}} \frac{\tau
    \abs{a}^{-1} \rho d\rho d\theta}{\sqrt{\rho - (\abs{a}-\tau)}}
  \leq 2\pi \tau \int_0^\tau \frac{d\rho}{\sqrt{\rho}} =
  4\pi\tau^{3/2} \leq 4\pi\tau\sqrt{\abs{a}}.
  \]
  Finally use the fact that $\DOI$ is the unit ball and thus
  $\abs{a}=1$ to conclude the claim.  
\end{proof}

We are now ready to prove stability for point source backscattering.

\begin{proof}[Proof of Theorem \ref{inverseThm}]
  Write $\tilde U^a = U_1^a-U_2^a$ and $\tilde q=q_1-q_2$. By the
  assumptions and Proposition \ref{prop:bndry2inside} we have
  \[
  \tau \tilde U^a(a,2\tau) = \frac{\tau}{32\pi^2\tau^2}
  \int_{\abs{x-a}=\tau} \tilde q(x) d\sigma(x) +
  \int_{\abs{x-a}\leq\tau} \tilde q(x) \tau k(x,\tau,a) dx
  \]
  for any $\tau>0$, in particular for $h<\tau<1$ which we shall assume
  now. By Proposition \ref{prop:differentiate} and the differentiation
  formula for moving regions (e.g. \cite{Evans} Appendix C.4) we get
  \begin{align*}
  &\partial_\tau \left( \tau \tilde U^a(a,2\tau) \right) =
    \frac{1-\tau}{8} \tilde q\big((1-\tau)a\big) + \frac{1}{4}
    E(a,\tau) \\ &\qquad + \int_{\abs{x-a}=\tau} \tilde q(x) \tau
    k(x,t,a) d\sigma(x) + \int_{\abs{x-a}\leq\tau} \tilde q(x)
    \partial_\tau (\tau k(x,\tau,a)) dx.
  \end{align*}

  By the Cauchy--Schwarz inequalities of $\R^4$ and the $L^2$-based
  function spaces $L^2(\{\abs{x-a}=\tau\})$ and
  $L^2(\{\abs{x-a}\leq\tau\})$ we have
  \begin{align*}
    &(1-\tau)^2 \abs{\tilde q\big((1-\tau)a\big)}^2 \leq 256
    \abs{\partial_\tau \left( \tau \tilde U^a(a,2\tau) \right)}^2 + 16
    \abs{E(a,\tau)}^2 \\ &\qquad + 256 \int_{\abs{x-a}=\tau}
    \abs{\tilde q(x)}^2 d\sigma(x) \int_{\supp \tilde q \cap
      \abs{x-a}=\tau} \abs{ \tau k(x,\tau,a) }^2 d\sigma(x) \\ &\qquad
    + 256 \int_{\abs{x-a}\leq\tau} \abs{ \tilde q(x)}^2 dx
    \int_{\supp\tilde q \cap \abs{x-a} \leq \tau} \abs{ \partial_\tau
      (\tau k(x,\tau,a)) }^2 dx
  \end{align*}
  Note that $q_1(x)=q_2(x)=0$ for $\abs{x-a}<h$. Also recall the
  estimates \eqref{kEst1} and \eqref{kEst2} for integrals of $k$ from
  Proposition \ref{prop:bndry2inside}. We can proceed then with
  \begin{align*}
  &(1-\tau)^2 \abs{\tilde q\big((1-\tau)a\big)}^2 \leq C_{M,h,\DOI}
    \Big( \abs{\partial_\tau \left( \tau \tilde U^a(a,2\tau)
      \right)}^2 + \abs{E(a,\tau)}^2 \\ &\qquad +
    \int_{\abs{x-a}=\tau} \abs{\tilde q(x)}^2 d\sigma(x) +
    \int_{\abs{x-a}\leq\tau} \abs{ \tilde q(x)}^2 dx \Big)
  \end{align*}
  since $\qnorm{q_1}, \qnorm{q_2} \leq \mathcal M$.

  Integrate the above estimate with $\int_{a\in\partial\DOI} \ldots
  d\sigma(a)$ and use the coordinate change of Lemma
  \ref{integrationChangeOfCoords}. Then write $\mathcal Q(r) =
  \int_{\abs{x}=r} \abs{\tilde q(x)}^2 d\sigma(x)$ and scale the
  integration variable on the left-hand side to get
  \begin{align}
    &\frac{\mathcal Q(1-\tau)}{C_{\mathcal M,h,\DOI}} \leq
    \int_{\abs{a}=1} \abs{ \partial_\tau(\tilde U^a(a,2\tau) ) }^2
    d\sigma(a) + \int_{\abs{a}=1} \abs{E(a,\tau)}^2 d\sigma(a) \notag
    \\ &\qquad\quad + \pi \int_{\abs{x}\geq 1 - \tau}
    \abs{\tilde{q}(x)}^2 \frac{\tau^2 + 2\tau -
      (1-\abs{x})^2}{\abs{x}} dx. \label{estimateAll}
  \end{align}

  Next, estimate $\abs{E(a,\tau)}^2$ using Proposition
  \ref{prop:differentiate}. Then change the order of integration using
  Lemma \ref{integrationChangeOfCoords}, switch to angular
  coordinates, and apply angular control \eqref{angularControl} to get
  \begin{align}
    &\int_{\abs{a}=1} \abs{E(a,\tau)}^2 d\sigma(a) \leq
    \frac{6\tau}{1-\tau} \sum_{i<j} \int_{\abs{x}\geq 1-\tau}
    \frac{\abs{\Omega_{ij} \tilde q(x)}}{\abs{x}
      \sqrt{\abs{x}-(1-\tau)}} d\sigma(x) \notag \\ &\qquad = \frac{6
      \tau}{1-\tau} \sum_{i<j} \int_{1-\tau}^1 \int_{\abs{x}=r}
    \frac{\abs{\Omega_{ij} \tilde q(x)}}{r \sqrt{r-(1-\tau)}}
    d\sigma(x) dr \notag \\ &\qquad \leq 6 S^2 \int_{1-\tau}^1
    \frac{\tau}{1-\tau} \frac{\mathcal Q(r)}{r\sqrt{r-(1-\tau)}}
    dr. \label{Eestim}
  \end{align}
  Similarly, the last term in \eqref{estimateAll} can be written as
  \begin{equation} \label{lastEstim}
    \ldots = \pi \int_{1-\tau}^1 \frac{\tau^2 + 2\tau - (1-r)^2}{r}
    \mathcal Q(r) dr.
  \end{equation}
  Finally, combine estimates \eqref{Eestim} and \eqref{lastEstim} to
  change \eqref{estimateAll} into
  \begin{align*}
    & \mathcal Q(1-\tau) \leq C_{\mathcal M,h,\DOI} \int_{\abs{a}=1}
    \abs{\partial_\tau (\tau \tilde U^a(a,2\tau))}^2 d\sigma(a)
    \\ &\qquad + C_{\mathcal M,h,\DOI} \int_{1-\tau}^1 \left( \frac{6
      S^2 \tau}{(1-\tau) r \sqrt{r-(1-\tau)}} + \pi \frac{\tau^2 +
      2\tau - (1-r)^2}{r} \right) \mathcal Q(r) dr
  \end{align*}
  which is valid for $0<\tau<1$.

  \bigskip
  Our next step is to prepare for Gr\"onwall's inequality. The
  inequality above can be written as
  \begin{equation} \label{almostGronwall}
    \varphi(\tau) \leq d(\tau) + \int_0^\tau \beta(\tau, s) \varphi(s)
    ds
  \end{equation}
  for $0<\tau<1$ where
  \[
  \varphi(\tau) = \mathcal Q(1-\tau), \qquad d(\tau) = C_{\mathcal
    M,h,\DOI} \int_{\abs{a}=1} \abs{\partial_\tau (\tau \tilde
    U^a(a,2\tau))}^2 d\sigma(a)
  \]
  and
  \[
  \beta(\tau, s) = C_{\mathcal M,h,\DOI} \left( \frac{6 S^2
    \tau}{(1-\tau) (1-s) \sqrt{\tau-s}} + \pi \frac{\tau^2 + 2\tau -
    s^2}{1-s} \right).
  \]
  Because of the singularities of $\beta$ we restrict
  \eqref{almostGronwall} to $0 < \tau \leq 1-\varepsilon$ for any
  given $\varepsilon>0$. We have $1-s \geq 1-\tau \geq \varepsilon >
  0$ and $\tau \leq 1$. In this situation we see easily that
  \[
  \beta(\tau,s) \leq \frac{6 C_{\mathcal M,h,\DOI} S^2}{\varepsilon^2
    \sqrt{\tau-s}} + \frac{3\pi C_{\mathcal
      M,h,\DOI}}{\sqrt{\varepsilon}\sqrt{\tau-s}} \leq \frac{6 S^2 +
    3\pi}{\varepsilon^2} \frac{C_{\mathcal M,h,\DOI}}{\sqrt{\tau-s}}.
  \]
  Denote $C_{S,\mathcal M,h,\DOI} = (6S^2 + 3\pi) C_{\mathcal
    M,h,\DOI}$.

  An application of Gr\"onwall's inequality (Lemma
  \ref{gronwallLemma}) implies
  \begin{equation} \label{finalProofEstimate}
    \varphi(\tau) \leq \left(1 + 2 C_{S,\mathcal M,h,\DOI}
    \varepsilon^{-2} \right) \sup_{0<\tau_0<1} d(\tau_0) \exp\left( 4
    C_{S,\mathcal M,h,\DOI}^2 \varepsilon^{-4} \tau \right)
  \end{equation}
  for $0 < \tau \leq 1-\varepsilon$. Now, given any $\tau \in (0,1)$
  we choose $\varepsilon>0$ such that $\tau \leq 1-\varepsilon$ and
  the right-hand side of the estimate above is minimized. These
  conditions are satisfied for $\varepsilon = 1-\tau$. The claim
  \eqref{condStab} follows after recalling that $\varphi(\tau) =
  \int_{\abs{x}=1-\tau} \abs{(q_1-q_2)(x)}^2 d\sigma(x)$ and applying
  simple estimates.

  \medskip
  Let us prove the norm estimate for $\tilde q = q_1-q_2$ over the
  whole $\DOI$ next. Rewrite \eqref{condStab} as
  \[
  \norm{\tilde q}_{L^2(\{\abs{x}=r\})} \leq \Lambda e^{\mathfrak
    C/r^4}
  \]
  where $\Lambda = \norm{U_1^a-U_2^a}$. Since $\qspace \hookrightarrow
  W^{1,\infty}(\DOI)$ and the potentials are supported in $\DOI$ we
  have the Lipschitz-norm estimate $\abs{\tilde q(x)} \leq \lvert{
    \tilde q(x+\ell \frac{x}{\abs{x}}) \rvert} + 2 \ell \mathcal M$
  for any $\ell\geq0$. Integration gives
  \[
  \norm{\tilde q}_{L^2(\{\abs{x}=r\})} \leq 2 \sqrt{4\pi} \mathcal M r
  \ell + \frac{r}{r+\ell} \Lambda e^{\mathfrak C/(r+\ell)^4}
  \]
  which we can estimate to
  \[
  \norm{\tilde q}_{L^2(\{\abs{x}=r\})} \leq 2 \sqrt{4\pi} \mathcal M
  \ell + \Lambda e^{\mathfrak C/\ell^4}
  \]
  because $0\leq r\leq 1$ and $\ell \geq 0$. The full domain estimate
  \eqref{condStabOrigin} follows from Lemma \ref{optimiseExponential}.

  \medskip The proof for $q_1-q_2$ radially symmetric proceeds as
  above until \eqref{almostGronwall}. Since in the condition of
  angular control \eqref{angularControl} we can assume that $S=0$, we
  have
  \[
  \beta(\tau,s) = C_{\mathcal M,h,\DOI} \pi
  \frac{\tau^2+2\tau-s^2}{1-s} \leq \frac{C'_{\mathcal M,h,\DOI}}{1-s}
  \]
  and so
  \[
  \frac{\varphi(\tau)}{C''_{\mathcal M,h,\DOI}} \leq
  \norm{U_1^a-U_2^a}^2 + \int_0^\tau \frac{\varphi(s)}{1-s} ds.
  \]
  This type of integral inequality implies
  \begin{align*}
    &\varphi(\tau) \leq C''_{\mathcal M,h,\DOI} \norm{U_1^a-U_2^a}^2
    \exp \left( \int_0^\tau \frac{C''_{\mathcal M,h,\DOI}}{1-s} ds
    \right) \\ &\qquad = C''_{\mathcal M,h,\DOI} \norm{U_1^a-U_2^a}^2
    (1-\tau)^{-2\alpha}
  \end{align*}
  for some $\alpha=\alpha(\mathcal M,h,\DOI)$ by Gr\"onwall's
  inequality. Note that here $\tau$ is allowed to be anywhere in the
  whole interval $(0,1)$ without any of the constants blowing
  up. Following the rest of the proof implies H\"older stability.
\end{proof}

\section{Technical tools}
We collect here some basic calculations and some well known theorems
so that we may refer to them without losing focus in the main proof.

\begin{lemma} \label{integrationChangeOfCoords}
  Let $f$ be a continuous function vanishing outside of $\DOI$ and let
  $\tau<1$ positive. Then
  \[
  \int_{\abs{a}=1} \int_{\abs{x-a}=\tau} f(x) d\sigma(x) d\sigma(a) =
  2\pi\tau \int_{\abs{x}\geq 1-\tau} \frac{f(x)}{\abs{x}} dx
  \]
  and
  \[
  \int_{\abs{a}=1} \int_{\abs{x-a}\leq\tau} f(x) dx d\sigma(a) = \pi
  \int_{\abs{x}\geq 1-\tau} \frac{f(x)}{\abs{x}} \big(\tau^2 -
  (1-\abs{x})^2\big) dx.
  \]
\end{lemma}

\begin{proof}
  The first equation was proven just before formula (2.10) in
  \cite{RU2}. The left-hand side of the second equation was shown to
  be equal to
  \[
  \int_{\abs{x}\leq1} f(x) \int_{\abs{a}=1}
  H(\tau^2-\abs{x-a}^2)d\sigma(a) dx
  \]
  therein too.

  The last equality follows by noting that the integral of the
  Heaviside function is just the area of the spherical cap arising
  from the intersection of $\abs{a}=1$ and $\abs{a-x}=\tau$. If
  $\abs{x}<1-\tau$ then this intersection is empty. Otherwise the area
  is seen to be $2\pi\cdot r \cdot h$, where $r=1$ is the radius of
  the sphere $\{\abs{a}=1\}$ and $h$ is the height of the cap along
  the ray $y\bar0$. Two applications of Pythagoras' theorem and some
  simple algebra imply that $h = (\tau^2 - (1-\abs{x})^2)/(2\abs{x})$
  and thus the final equality is proven.
\end{proof}

\begin{lemma} \label{gronwallLemma}
  Let $b>a$ and $d\colon (a,b) \to \R$ be bounded and
  measurable. Moreover let $\beta\colon (\tau,s) \mapsto
  \beta(\tau,s)$ be measurable whenever $\tau,s \in (a,b)$ and $s <
  \tau$.  Moreover let it satisfy
  \[
  \beta(\tau,s) \leq \frac{C}{\sqrt{\tau-s}}
  \]
  for some $C<\infty$ whenever $s<\tau$.

  If $\varphi\colon (a,b) \to \R$ is a non-negative integrable
  function that satisfies the integral inequality
  \begin{equation}
    \varphi(\tau) \leq d(\tau) + \int_a^\tau \beta(\tau,s) \varphi(s)
    ds
  \end{equation}
  for almost all $\tau \in (a,b)$, then
  \[
  \varphi(\tau) \leq (1 + 2 C \sqrt{b-a}) \sup_{a<\tau_0<b} d(\tau_0)
  e^{4C^2 \tau}.
  \]
\end{lemma}

\begin{proof}
  First of all note that since $\varphi \geq 0$, we may estimate
  $\beta$ from above in the integral, and see that the former
  satisfies
  \[
  \varphi(\tau) \leq d(\tau) + C \int_a^\tau
  \frac{\varphi(s)}{\sqrt{\tau-s}} ds
  \]
  for almost all $\tau$.

  Next bootstrap the above by estimating $\varphi$ inside the integral
  using that same inequality. Then
  \[
  \varphi(\tau) \leq d(\tau) + C \int_a^\tau
  \frac{d(s)}{\sqrt{\tau-s}} ds + C^2 \int_a^\tau \int_a^s
  \frac{\varphi(s')}{\sqrt{\tau-s} \sqrt{s-s'}} ds' ds.
  \]

  The double integral is estimated as follows: $\int_a^\tau \int_a^s
  \ldots ds' ds = \int_a^\tau \int_{s'}^\tau \ldots ds ds'$, and then
  we are left to estimate $\int_{s'}^\tau ds / \sqrt{\tau-s}
  \sqrt{s-s'}$. To do that split the interval $(s',\tau)$ into two
  equal parts by the midpoint $s = (\tau+s')/2$. In the interval
  $s\in(s',(\tau+s')/2)$ we have $1/\sqrt{\tau-s} \leq
  \sqrt{2/(\tau-s')}$ and $\int_{s'}^{(\tau+s')/2} ds/\sqrt{s-s'} =
  \sqrt{2(\tau-s')}$. Their product is equal to $2$. The same
  deduction works in the second interval. Hence
  \[
  \int_{s'}^\tau \frac{ds}{\sqrt{\tau-s}\sqrt{s-s'}} \leq 4
  \]
  indeed and
  \[
  \varphi(\tau) \leq d(\tau) + C \int_a^\tau
  \frac{d(s)}{\sqrt{\tau-s}} ds + 4C^2 \int_a^\tau \varphi(s') ds'
  \]
  follows.

  The first two terms above have an upper bound
  \[
  (1 + 2 C \sqrt{b-a}) \sup_{a<\tau_0<b} d(\tau_0)
  \]
  because $\int_a^\tau ds/\sqrt{\tau-s} = 2\sqrt{\tau-a} \leq
  2\sqrt{b-a}$. Grönwall's inequality implies the final claim: If
  $\varphi(\tau) \leq C_1 + C_2 \int_0^\tau \varphi(s) ds$ for
  $\tau\geq0$ where $\varphi\geq0$ then $\varphi(\tau) \leq C_1
  \exp(C_2\tau)$. This follows for example from \mbox{Appendix~B.2.j}
  in \cite{Evans} and some algebra. Note however that the integral
  form of Gr\"onwall's inequality in \mbox{Appendix~B.2.k} of
  \cite{Evans} is weaker than this one.
\end{proof}

\begin{lemma} \label{optimiseExponential}
  Let $f \colon \R_+ \to \R$ be a positive function satisfying
  \[
  f(\ell) \leq A\ell + \Lambda e^{\mathfrak C/\ell^4}
  \]
  for some $\Lambda < \infty$ and any $\ell$ in its domain. Then if
  $0<\Lambda<e^{-1}$ we have
  \[
  f(\ell_0) \leq \frac{A (2\mathfrak C)^{1/4} + 2}{\left( \ln
    \frac{1}{\Lambda} \right)^{1/4}}
  \]
  where $\ell_0^4 = \mathfrak C / (\ln \frac{1}{\sqrt{\Lambda}})$. If
  $\Lambda \geq e^{-1}$ then we have the linear estimate
  \[
  f(\ell_0) \leq (A \mathfrak C^{1/4} + 1) e \Lambda.
  \]
  for $\ell_0^4 = \mathfrak C$.
\end{lemma}

\begin{proof}
  Since $\Lambda < e^{-1}$ the choice of $\ell_0$ is proper. Moreover
  we see immediately that
  \[
  f(\ell_0) \leq \frac{A (2 \mathfrak C)^{1/4}}{(\ln
    \frac{1}{\Lambda})^{1/4}} + \sqrt{\Lambda}.
  \]
  Recall the elementary inequality $\ln \frac{1}{a} \leq \frac{1}{b}
  a^{-b}$ for $b>0$ and $0<a<e^{-1}$. Set $b=2$ and $a = \Lambda$ to
  see that
  \[
  \sqrt{\Lambda} \leq \frac{2}{\ln \frac{1}{\Lambda}} \leq
  \frac{2}{(\ln \frac{1}{\Lambda})^{1/4}}
  \]
  since $\ln \frac{1}{\Lambda} > 1$ then. The first claim follows. The
  second claim is elementary.
\end{proof}

The following is from personal communication with Rakesh.
\begin{lemma} \label{ellipticIntegral}
  Let $p\colon \R\to\R$ be a measurable function. Then, given any time
  $t\geq0$ and position $x\in\R^n$ with $t \geq \abs{x}$, we have
  \[
  \abs{y}+\abs{x-y}\leq t \quad \Longleftrightarrow \quad
  (t-\abs{y})^2 - \abs{x-y}^2 \geq 0
  \]
  and
  \begin{align*}
  &\int_{\abs{y}+\abs{x-y} \leq t}
    \frac{p\big((t-\abs{y})^2-\abs{x-y}^2\big)}{\abs{y}} dy
    \\ &\qquad= \int_{\abs{w}\leq\frac{1}{2}\sqrt{t^2-\abs{x}^2}}
    \frac{p\big((\sqrt{t^2-\abs{x}^2}-\abs{w})^2-\abs{w}^2\big)}{\abs{w}}
    dw.
  \end{align*}
\end{lemma}

\begin{proof}
The first claim follows from the triangle inequality applied to a
triangle with vertices $x$, $y$ and $\bar0$: $t-\abs{y}+\abs{x-y} \geq
\abs{x}-\abs{y}+\abs{x-y} \geq 0$, so we may multiply the inequality
\[
t-\abs{y}-\abs{x-y} \geq 0
\]
by the former without changing sign.

Let $p_+(r) = p(r)$ for $r\geq0$ and $p_+(r) = 0$ for $r<0$. Denote
the left-hand side integral in the statement by $I$. Then
\begin{align*}
  I&= \int_{\R^3}
  \frac{p_+\big((t-\abs{y})^2-\abs{x-y}^2\big)}{\abs{y}} dy \\ &=
  \int_{\R^3} \int_{-\infty}^\infty \frac{\delta(s-\abs{y})}{\abs{y}}
  p_+\big((t-\abs{y})^2-\abs{x-y}^2\big) ds dy \\ &=
  \int_{-\infty}^\infty \int_{\R^3} \frac{\delta(s-\abs{y})}{\abs{y}}
  p_+\big((t-\abs{y})^2-\abs{x-y}^2\big) dy ds \\ &= 2
  \int_{-\infty}^\infty \int_{\R^3} \delta(s^2-\abs{y}^2)
  p_+\big((t-\abs{y})^2-\abs{x-y}^2\big) dy ds.
\end{align*}

Let $L_1\colon \R^3\to\R^3$ be a rotation taking $x \mapsto
(\abs{x},0,0)$. Let it map $y \mapsto y'$. Then $dy = dy'$ and so
\[
I = 2 \int_{-\infty}^\infty \int_{\R^3} \delta(s^2-\abs{y'}^2)
p_+\big((t-\abs{y'})^2-\abs{L_1x-y'}^2\big) dy' ds.
\]
Next let $(s,y') \mapsto z \in \R^4$ be the Lorentz transformation
given by
\[
z_0 = \frac{ts-\abs{x}y_1'}{\sqrt{t^2-\abs{x}^2}}, \quad z_1 = \frac{t
  y_1' - \abs{x} s}{\sqrt{t^2-\abs{x}^2}}, \quad z_2 = y_2, \quad z_3
= y_3.
\]
It is a trivial matter to see that $dz = dy'ds$ and the following
identities
\[
z_0^2 - z_1^2 = s^2-y_1'^2, \qquad \big(\sqrt{t^2-\abs{x}}-z_0\big)^2
- z_1^2 = (t-s)^2 - (\abs{x}-y_1')^2.
\]

Finally, denoting $\abs{z}^2 = z_1^2+z_2^2+z_3^2$ and $z\cdot z= z_0^2
- \abs{z}^2$, we have
\begin{align*}
  I&= 2\int_{\R^4} \delta(z\cdot z)
  p_+\big((\sqrt{t^2-\abs{x}^2}-z_0)^2 - \abs{z}^2\big) dz \\ &=
  \int_{-\infty}^\infty \int_{\R^3}
  \frac{\delta(z_0-\abs{z})}{\abs{z}}
  p_+\big((\sqrt{t^2-\abs{x}^2}-z_0)^2 - \abs{z}^2\big)
  dz_1dz_2dz_3dz_0 \\ &= \int_{\R^3}
  \frac{p_+\big((\sqrt{t^2-\abs{x}^2}-\abs{z})^2
    -\abs{z}^2\big)}{\abs{z}} dz_1dz_2dz_3\\ &= \int_{\R^3}
  \frac{p_+\big((\sqrt{t^2-\abs{x}^2}-\abs{w})^2-\abs{w}^2\big)}{\abs{w}}
  dw
\end{align*}
which implies the claim since
$(\sqrt{t^2-\abs{x}^2}-\abs{w})^2-\abs{w}^2 \geq 0$ if and only if
$\sqrt{t^2-\abs{x}^2}-\abs{w}-\abs{w} \geq 0$.
\end{proof}

\subsection*{Acknowledgements}
I am indebted to Rakesh for the many discussions that led me to
understanding the Goursat problem and how to show the well-posedness
of the point source problem. Without his help this important part of
the paper would have taken many more months to complete. In addition I
would like to thank the anonymous referees and their comments. This
led to the realization that radially symmetric potentials have a
better stability estimate.

\end{document}